\documentclass[a4paper]{article}
\usepackage[left=3.5cm,right=3.0cm,top=2.5cm,bottom=2.5cm]{geometry}
\usepackage{xcolor}
\usepackage{graphicx}
\usepackage{cite}
\usepackage[thmmarks]{ntheorem}
\usepackage{amsmath,amssymb,amscd,amsxtra,amsfonts}

\newcommand\keywordsname{Key words}
\newcommand\AMSname{AMS subject classifications}

\newenvironment{@abssec}[1]
{\if@twocolumn
\section*{#1}%
\else
\vspace{.05in}\footnotesize
\parindent .2in
{\upshape\bfseries #1. }\ignorespaces
\fi}

{\if@twocolumn\else\par\vspace{.1in}\fi}

\newenvironment{keywords}{\begin{@abssec}{\keywordsname}}{\end{@abssec}}

\providecommand{\Div}{\operatorname{div}}          % Divergence
\providecommand{\curl}{\operatorname{{\bf curl}}}  % Standard curl operator
\providecommand*{\Dist}[2]{\operatorname{dist}({#1};{#2})}   % distance
\providecommand*{\Dist}[2]{\Dist{#1}{#2}}
\newcommand{\Vb}{{\mathbf{b}}}

\newcommand{\Ve}{{\mathbf{e}}}

\newcommand{\Vr}{{\mathbf{r}}}

\newcommand{\Vx}{{\mathbf{x}}}

\newcommand{\Ba}{{\boldsymbol{a}}}

\newcommand{\Bf}{{\boldsymbol{f}}}
\newcommand{\Bg}{{\boldsymbol{g}}}

\newcommand{\Bn}{{\boldsymbol{n}}}

\newcommand{\Bv}{{\boldsymbol{v}}}
\newcommand{\Bw}{{\boldsymbol{w}}}
\newcommand{\Bx}{{\boldsymbol{x}}}

\newcommand{\VC}{{\mathbf{C}}}
\newcommand{\VD}{{\mathbf{D}}}

\newcommand{\BA}{{\boldsymbol{A}}}
\newcommand{\BB}{{\boldsymbol{B}}}

\newcommand{\BE}{{\boldsymbol{E}}}
\newcommand{\BF}{{\boldsymbol{F}}}

\newcommand{\BH}{{\boldsymbol{H}}}

\newcommand{\BJ}{{\boldsymbol{J}}}
\newcommand{\BK}{{\boldsymbol{K}}}
\newcommand{\BL}{{\boldsymbol{L}}}
\newcommand{\BM}{{\boldsymbol{M}}}

\newcommand{\BP}{{\boldsymbol{P}}}

\newcommand{\BX}{{\boldsymbol{X}}}

\newcommand{\varphibf}{\boldsymbol{\varphi}}

\newcommand{\Ct}{\mathcal{T}}

\newcommand{\bbR}{\mathbb{R}}

\newcommand*{\N}[1]{\left\|{#1}\right\|}     % Double bar norm
\newcommand*{\Lp}[2][\defaultdomain]{L^{#2}({#1})}
\newcommand*{\Lpv}[2][\defaultdomain]{\BL^{#2}({#1})}
\newcommand*{\NLp}[3][\defaultdomain]{\N{#2}_{\Lp[#1]{#3}}}

\newcommand*{\Ltwo}[1][\defaultdomain]{\Lp[#1]{2}}
\newcommand*{\Ltwov}[1][\defaultdomain]{\Lpv[#1]{2}}
\newcommand*{\NLtwo}[2][\defaultdomain]{\NLp[#1]{#2}{2}}

\newcommand*{\Hmv}[2][\defaultdomain]{\BH^{#2}({#1})}

\newcommand*{\Honev}[1][\defaultdomain]{\Hmv[#1]{1}}

\newcommand*{\Hcurl}[1][\defaultdomain]{\boldsymbol{H}(\curl,{#1})}
\newcommand*{\bHcurl}[2][\defaultdomain]{\boldsymbol{H}_{#2}(\curl,{#1})}
\newcommand*{\zbHcurl}[1][\defaultdomain]{\bHcurl[#1]{0}}

\newcommand*{\NHcurl}[2][\defaultdomain]{\N{#2}_{\Hcurl[#1]}}

\newcommand{\D}{\mathrm{d}}

\newcommand{\be}{\begin{eqnarray}}
\newcommand{\ee}{\end{eqnarray}}

\newcommand{\ben}{\begin{eqnarray*}}
\newcommand{\een}{\end{eqnarray*}}

\newtheorem{theorem}{\sc Theorem}[section]

\newtheorem{algorithm}[theorem]{\sc Algorithm}
\newtheorem{example}[theorem]{\sc example}

\theoremstyle{nonumberplain}
\theoremsymbol{$\Box$}
\newtheorem{proof}{\textbf{Proof.}}

\title{A novel divergence-free Finite Element Method for the MHD Kinematics equations using Vector-potential}
\author{Lingxiao Li
\thanks{Academy of Mathematics and Systems Science, Chinese Academy of Sciences, Beijing, 100190, China;
School of Mathematical Science, University of Chinese Academy of Sciences, Beijing, 100049, China (lilingxiao@lsec.cc.ac.cn).}}
\begin{document}
\date{\textit{Dedicated to my family and in memory of my Grandma.}}
\maketitle

\begin{abstract}
  We propose a new mixed finite element method for the three-dimensional steady magnetohydrodynamic (MHD) kinematics equations for which the velocity of the fluid is given.
  Although prescribing the velocity field leads to a simpler model than full MHD equations, its conservative and efficient numerical methods are still active research topic.
  The distinctive feature of our discrete scheme is that the divergence-free conditions for current density and magnetic induction are both satisfied.
  To reach this goal, we use magnetic vector potential to represent magnetic induction and resort to H(div)-conforming element to discretize the current density.
  We develop an preconditioned iterative solver based on a block preconditioner for the algebraic systems arising from the discretization.
  Several numerical experiments are implemented to verify the divergence-free properties,
  the convergence rate of the finite element scheme and the robustness of the preconditioner.
\end{abstract}


\section{Introduction}
\label{sec:introduction}

Magnetohydrodynamics (MHD) has broad applications in our real world.
It describes the interaction between electrically conducting fluids
and magnetic fields. It is used in industry to heat, pump, stir, and
levitate liquid metals. MHD model also governs the terrestrial magnetic filed maintained by fluid motion in the earth
core and the solar magnetic field which generates sunspots and solar
flares\cite{dav01}. The general full MHD model consists of the
Navier-Stokes equations and the quasi-static Maxwell equations.
The magnetic field influences the momentum of the fluid
through Lorentz force, and conversely, the motion of fluid influences the magnetic field through Faraday's law. In this paper,
we are studying the divergence-free finite element method and efficient iterative solver for the MHD kinematics equations
\begin{subequations}\label{eq:mhd}
\begin{align}
\partial_t\BB + \curl\BE &= \mathbf{0}  \hspace{0.5mm}\qquad  \hbox{in}\;\;(0,T]\times\Omega,   \label{eq:Faraday}\\
\curl\BH &= \BJ                         \qquad \hbox{in}\;\;[0,T]\times\Omega,                  \label{eq:ns}\\
\sigma(\BE+\Bw\times\BB) &= \BJ         \qquad \hbox{in}\;\;(0,T]\times\Omega,                  \label{eq:OhmLaw}\\
\Div\BJ =0,   \quad\Div\BB&=0           \hspace{1.0mm}\qquad \hbox{in}\;\;[0,T]\times\Omega.    \label{eq:div}
\end{align}
\end{subequations}
where $\Bw$ is the known velocity field (\textbf{here $\Bw$ is not necessary divergence-free, so the method developed in
this paper can be extended to the compressible MHD equations}),
$\BB$ is the magnetic flux density or the magnetic field provided with constant permeability.
We assume that $\Omega$ is a bounded, simply-connected, and Lipschitz polyhedral domain with boundary $\Gamma=\partial\Omega$.
When $\Bw \equiv 0$ in the domain, this is indeed corresponding to the classical eddy currents model\cite{buff2000, hip02, jiang2012}.
The equations in \eqref{eq:mhd} are complemented with the following constitutive equation
\begin{equation}
\BB=\mu\BH
\end{equation}

Though prescribing the velocity field leads to a simpler model than full MHD equations,
its conservative and efficient numerical methods are still active research topic.
The equations of MHD kinematics
have interest applications in the field of dynamo theory, where the Lorentz force $\BJ\times\BB$ is assumed to be small
compared to inertia\cite{liu2007, mini2005}. In such situation, the electromagnetics have little influence
on the fluid and the velocity could be prescribed. Then we can investigate the variation of the magnetic field caused by
the flow profile. Such applications contain MHD generators, dynamo of the sun, brine and the geodynamo.
It is interesting that whether a given $\Bw$ can sustain dynamo action.
At last, as an important block of the full MHD equations, just as the Navier-Stokes equations, more efficient algorithm
for the sub-block of course will lead to more efficient method for MHD equations.

There are extensive papers in the literature to study numerical
solutions of MHD equations (cf. e.g. \cite{ger06,gre10,gun91,ni12,ma16, ni07-1,ni07-2}
and the references therein). In \cite{gun91}, Gunzburger et al
studied well-posedness and the finite element
method for the stationary incompressible MHD
equations. The magnetic field is discretized by the $\Honev$-conforming
finite element method.
We also refer to \cite{ger06} for a systematic
analysis on finite element methods for incompressible MHD
equations. When the domain has re-entrant angle, the magnetic field
may not be in $\Honev$. It is preferable to use noncontinuous finite
element functions to approximate $\BB$, namely, the so-called edge
element method\cite{hip02,ne86}. In 2004, Sch\"{o}tzau\cite{sch04}
proposed a mixed finite element method to solve the stationary
incompressible MHD equations where edge elements are used to solve the
magnetic field. For the finite element model above, only the divergence-free condition
of the current density $\BJ_h$ can be precisely satisfied, because they using the definition
$\BJ = \curl\BB$.

In recent years, exactly divergence-free approximations for $\BJ$ and $\BB$ have attracted
more and more interest in numerical simulation. For the current density $\BJ$ we would like to mention the
current density-conservative finite volume methods of Ni et al. for the inductionless MHD model on both structured and unstructured
grids\cite{ni12, ni07-1, ni07-2}. In \cite{hu17, ma16}, the authors discretize the electric field $\BE$ by N\'{e}d\'{e}lec's edge elements
and discretize the magnetic induction $\BB$ by Raviart-Thomas elements. In such way, $\Div\BB_h$ is precisely zero in the discrete level.
Recently in\cite{hip2017}, Hiptmair, Li, Mao and Zheng proposed a novel mixed finite element method in 3D using magnetic vector $\BA$
with temporal gauge to ensure the divergence-free condition
for $\BB_h$ by $\BB_h = \curl\BA_h$. Very recently, Li, Ni and Zheng devised a electron-conservative mixed finite element method
for the inductionless MHD equations where an applied magnetic field is prescribed\cite{li2017}, for which $\|\Div\BJ_h\| = 0$ can be satisfied
precisely. They also developed a robust block preconditioner for
the linear algebraic systems. However, none of the work mentioned above can achieve the goal
\[\Div\BJ_h = 0,\quad \Div\BB_h = 0\]
precisely at the same time in the finite element framework for the induced MHD equations.

Instead of solving for $\BB$ and $\BE$, we solve for the magnetic vector potential
$\BA$ and scalar electric potential $\phi$ such that
\begin{equation}\label{eq:trans}
\BB = \curl\BA,\quad \BE = -\partial_t\BA - \nabla\phi
\end{equation}
with the Coulomb's gauge rather than temporal gauge in\cite{hip2017}
\[\Div\BA = 0\]
Using the magnetic vector potential $\BA$ and scalar potential $\phi$, we develop a special mixed finite element method
that ensures exactly divergence-free approximations of the magnetic induction and current density.
To the best knowledge of the author, the finite element scheme which can preserve the divergence-free conditions for both
the current density and magnetic induction has never been reported in the literature. So the work in our paper could be seen as
a first step to the fully-conservative method for the fully-coupled MHD equations with induction. Another objective of this paper is to propose a preconditioned iterative
method to solve the algebraic systems associated with the proposed fully divergence-free FE scheme.

The paper is organized as follows:
In section 2, we introduce a new dimensionless model using magnetic vector potential $\BA$ and electrical scalar potential $\phi$.
The perfect conducting boundary conditions are given for the new model for practical implementation.
In section 3, we introduce a variational formulation for the MHD kinematics equations and give the energy law for the continuous problem.
In Section 4, we present a novel finite element scheme which can preserve the divergence-free properties precisely.
And a block preconditioner is given to develop a preconditoned FGMRES algorithm for the solving of the algebraic systems.
In section 5, three numerical experiments are conducted to verify the conservation of the discrete field or divergence-free characteristic,
the convergence rate of the mixed finite element method, and to demonstrate the optimality and the robustness of the iterative solver.
In section 6, some conclusions and further investigations are pointed out.

Throughout the paper we denote vector-valued quantities by boldface notation, such as $\Ltwov:=(\Ltwo)^3$.

\section{A dimensionless model using vector potential}
Using the transformation \eqref{eq:trans} and the Coulomb's gauge condition, one can obtain the following new PDE systems
\begin{subequations}\label{eq:new-mhd}
\begin{align}
\sigma^{-1}\BJ + \nabla\phi - \Bw\times\curl\BA +\partial_t\BA = \mathbf{0},&\quad \mathrm{in}~~(0,T]\times\Omega
\label{current}\\
\Div\BJ = 0, &\quad \mathrm{in}~~[0,T]\times\Omega
\label{conserv}\\
-\mu\BJ + \curl\curl\BA = \mathbf{0},&\quad \mathrm{in}~~(0,T]\times\Omega
\label{vectorp}\\
\Div\BA = 0, &\quad \mathrm{in}~~[0,T]\times\Omega
\label{aguage}
\end{align}
\end{subequations}
Let $L$, $t_0$, $B_0$ and $\sigma_0$ be the characteristic length, characteristic time, characteristic magnetic flux density and reference conductivity respectively.
Thus $u_0=L/t_0$ will be the characteristic fluid velocity. Due to the relation $\BB=\curl\BA$, we can introduce the characteristic vector potential
\begin{equation*}
A_0 = B_0L
\end{equation*}

If we make the following scaling of the variables
\begin{equation}\label{mhdscal}
t\leftarrow t\frac{1}{t_0},~\Bx\leftarrow \Bx\frac{1}{L},~\Bw\leftarrow \Bw\frac{1}{u_0},~\BJ\leftarrow\BJ\frac{1}{\sigma_0 u_0 B_0},~\phi\leftarrow \phi\frac{1}{u_0B_0L},
~\sigma\leftarrow \sigma\frac{1}{\sigma_0},~\BA\leftarrow \BA\frac{1}{A_0}
\end{equation}
and introduce the following dimensionless parameter
\begin{equation*}
\textsf{Rm} = \mu\sigma_0 L u_0
\end{equation*}
For the eddy currents model, $\Bw$ is identically zero, so we only need to make the following scaling
\begin{equation}\label{eddyscal}
t\leftarrow t\frac{1}{t_0},~\Bx\leftarrow \Bx\frac{1}{L},~\BJ\leftarrow\BJ\frac{t_0}{\sigma_0B_0L},~\phi\leftarrow \phi\frac{t_0}{B_0L^2},
~\sigma\leftarrow \sigma\frac{1}{\sigma_0},~\BA\leftarrow \BA\frac{1}{A_0}
\end{equation}
and define
\[\textsf{Rm} = \frac{\mu\sigma_0 L^2}{t_0}\]
Then we can get the desired dimensionless formulation
\begin{subequations}\label{eq:new-mhd-dimen}
\begin{align}
\sigma^{-1}\BJ + \nabla\phi - \Bw\times\curl\BA +\partial_t\BA = \Bf,&\quad \mathrm{in}~~(0,T]\times\Omega
\label{NewCurrent}\\
\Div\BJ = 0, &\quad \mathrm{in}~~[0,T]\times\Omega
\label{NewConserv}\\
-\BJ + \textsf{Rm}^{-1}\curl\curl\BA = \Bg,&\quad \mathrm{in}~~(0,T]\times\Omega
\label{NewCectorp}\\
\Div\BA = 0, &\quad \mathrm{in}~~[0,T]\times\Omega
\label{NewAguage}
\end{align}
\end{subequations}
where $\Bf = \Bg = \mathbf{0}$ generally with non-zero boundary conditions for the physical field. But for analytic test $\Bf$ and $\Bg$ may be nonzero,
so we retain the symbols $\Bf$ and $\Bg$ for general purpose.
The steady model of \eqref{eq:new-mhd-dimen} reads
\begin{subequations}\label{eq:steady-mhd}
\begin{align}
\sigma^{-1}\BJ + \nabla\phi - \Bw\times\curl\BA  = \Bf,&\quad \mathrm{in}~~\Omega\label{eq:steady-current}\\
\Div\BJ = 0, &\quad \mathrm{in}~~\Omega\\
-\BJ + \textsf{Rm}^{-1}\curl\curl\BA = \Bg,&\quad \mathrm{in}~~\Omega\label{eq:steady-A}\\
\Div\BA = 0, &\quad \mathrm{in}~~\Omega
\end{align}
\end{subequations}
We refer to the new model \eqref{eq:new-mhd} and \eqref{eq:steady-mhd} for the Kinematics equations \eqref{eq:mhd}
as $\mathrm{CPP^3}$-model (Current-Potential-vector Potential-model).

Generally, in the dimensionless systems $\sigma$ has the magnitude of $O(1)$.
The system of equations \eqref{eq:new-mhd-dimen} and \eqref{eq:steady-mhd} are still needed to be complemented with suitable boundary conditions.
Because it is a particularly important case when the material on one side of the interface is a perfect conductor.
So also for concise representation, we focus our attention on the perfectly electrically conductive (PEC) boundary conditions.
We will give general PEC boundary conditions on the interface between two different kinds of fluid, and then decide the special case used in the new model
for practical purpose.
In practical physical systems, if the boundary is perfect conducting, the electrical potential $\phi$ should be identical everywhere on the boundary.
So we can prescribe zero Dirichlet boundary conditions for $\phi$ in \eqref{eq:new-mhd-dimen}.

Next we still need to determine the boundary condition for $\BA$.
Note that for the original systems
\eqref{eq:mhd} the PEC boundary conditions means that $\sigma\rightarrow\infty$. From generalized Ohm's law \eqref{eq:OhmLaw},
we see heuristically that if the current density $\BJ$ is to remain bounded then
\begin{equation}\label{PEC1}
\BE^{*} := \BE+\Bw\times\BB \rightarrow \mathbf{0} \qquad\hbox{on}\;\;\Gamma.
\end{equation}
Let $\BE_1, \BE_2, \Bw_1, \Bw_2, \BB_1$ and $\BB_2$ be the limiting value of the physical field
as the interface $\Gamma$ is approached.
Then \eqref{PEC1} suggests that in a perfect conductor the total electrical field vanishes
\[\BE^{*}_2 := \BE_2+\Bw_2\times\BB_2 = \mathbf{0}\]
Then $\BE_2 = -\Bw_2\times\BB_2$ holds outside the domain $\Omega$. Because the tangential component of $\BE$ across $\Gamma$ is continuous, we must have
\begin{equation}\label{PEC2}
\BE_1\times\Bn = \BE_2\times\Bn = \Bn \times (\Bw_2\times\BB_2)
\end{equation}
Using identity $\Bn \times (\Bw_2\times\BB_2) = \Bw_2(\Bn\cdot\BB_2) - \BB_2(\Bn\cdot\Bw_2)$, and the continuity for velocity and magnetic induction
\[\BB_1\cdot\Bn = \BB_2\cdot\Bn,\quad \Bw_1\cdot\Bn = \Bw_2\cdot\Bn\]
for two-fluid case. Based on \eqref{PEC2} one obtains
\begin{equation}\label{PEC3}
\BE_1\times\Bn = \Bw_2(\BB_1\cdot\Bn) - \BB_2(\Bw_1\cdot\Bn) \qquad\hbox{on}\;\;\Gamma.
\end{equation}
For a general static solid wall, we must have $\Bw_2 \equiv \mathbf{0}$ and $\Bw_1\cdot\Bn \equiv 0$ on the whole boundary.
Thus the PEC boundary condition for the MHD kinematics equations \eqref{eq:mhd} reads
\begin{equation}\label{PEC4}
\BE\times\Bn = \mathbf{0} \qquad\hbox{on}\;\;\Gamma.
\end{equation}
In summary, given the applied magnetic induction $\BB_s$, the complete PEC boundary conditions reads
\begin{equation}\label{PEC5}
\phi = 0,~~\BE\times\Bn = \mathbf{0},~~\BB\cdot\Bn = \BB_s\cdot\Bn \qquad\hbox{on}\;\;\Gamma:=\partial\Omega.
\end{equation}
if the normal component of $\Bw$ vanishes on $\Gamma$.
For non-zero $\BB_s$ we can reduce the problem with $\BB\cdot\Bn = \BB_s\cdot\Bn$ to a problem with homogeneous boundary conditions $\BB\cdot\Bn=0$, if we make a "lifting"
procedure $\BB \Leftarrow \BB - \BB_s$ and with a source term $\Bf \Leftarrow \Bf + \Bw\times\BB_s$ in \eqref{NewCurrent} and \eqref{eq:steady-current}.
Denoting the corresponding vector potential for $\BB_s$ by $\BA_s$,
then $\BA_s$ is indeed the initial condition for $\BA$.

Let $\BA \Leftarrow \BA-\BA_s$,
and then $\Bg \Leftarrow \Bg - \textsf{Rm}^{-1}\curl\curl\BA_s = \Bg - \textsf{Rm}^{-1}\curl\BB_s$ in \eqref{NewCectorp} and
\eqref{eq:steady-A}. After "lifting", the initial condition for $\BA$ becomes zero.
Recall that
\begin{equation}
\BB\cdot\Bn = \left.(\curl\BA)\right|_{\Gamma}\cdot\Bn \triangleq \Div_{\Gamma}(\BA\times\Bn)
\end{equation}
If we pose the boundary conditions $\BA\times\Bn = \mathbf{0}$ we naturally get $\BB\cdot\Bn = 0$. Moreover from the transformation \eqref{eq:trans}
and conditions \eqref{PEC5} we then have
\begin{equation*}
\left.\BE \times \Bn \right|_{\Gamma} = \left(-\partial_t\BA)\times\Bn\right|_{\Gamma} = \textbf{0}
\end{equation*}
This indicates
\begin{equation}\label{eq:bdele1}
\left.\BA(t)\times\Bn\right|_{\Gamma} =  \mathbf{0},\quad \forall t \in(0,T]
\end{equation}
Now our PEC boundary conditions for \eqref{eq:new-mhd-dimen} and \eqref{eq:steady-mhd} are stated as follows
\begin{equation}\label{eq:bc}
\phi = 0,\quad \BA\times\Bn = \mathbf{0} \qquad\hbox{on}\;\;\Gamma.
\end{equation}
with $\Bf \Leftarrow \Bf + \Bw\times\BB_s$ and $\Bg \Leftarrow \Bg - \textsf{Rm}^{-1}\curl\BB_s$, provided $\Bw\cdot\Bn \equiv 0$ on $\Gamma$.

Or more generally
\begin{equation}\label{eq:bc1}
\phi = \phi_w,\quad \BA\times\Bn = \BA_w\times\Bn \qquad\hbox{on}\;\;\Gamma.
\end{equation}
In the following, we will focus on the steady version \eqref{eq:steady-mhd} and devise a divergence-free finite element method.
For convenience we consider the boundary conditions \eqref{eq:bc} with non-zero right-hand sides $\Bf$ and $\Bg$.
We will use electrical resistivity $\eta$ instead of $\sigma^{-1}$ and $\nu_m$ instead of $\textsf{Rm}^{-1}$ in some places.
\vspace{5mm}

\textbf{Remark 1.} It suffices to decide $\BA_s$ by solving the following problem
\begin{align}
\curl\BA_s= \BB_s,\quad \Div\BA_s =0 &\quad\hbox{in}\;\;\Omega, \label{eq:A0}\\
\BA_s\times\Bn = 0 &\quad\hbox{on}\;\;\Gamma. \label{bc:A0}
\end{align}
provided $\BB_s\cdot\Bn = 0$ on $\Gamma$. For steady systems \eqref{eq:steady-mhd}, due to $\BE = -\nabla\phi$ and zero boundary condition for $\phi$
we naturally obtain $\left.\BE\times\Bn\right|_\Gamma = \mathbf{0}$. We also refer the reader to the work\cite{four06} by Fournier which is concerned
with the electrically insulating boundary conditions for incompressible MHD equations. The author uses Fourier-spectral element to discretize
$\BA$ rather than finite element method.

%%%%%%%%%%%%%%%%%%%%%%
\section{Variational formulation for the MHD kinematics}
First we introduce some Hilbert spaces and Sobolev norms used in
this paper. Let $L^2(\Omega)$ be the usual Hilbert space of square
integrable functions equipped with the following inner product and
norm:
\begin{eqnarray}
(u,v):=\int_{\Omega}u(\Bx)\,v(\Bx)\D\Bx \quad \hbox{and} \quad
\NLtwo{u}:=(u,u)^{1/2}. \nonumber
\end{eqnarray}

Define $H^m(\Omega):=\{v\in L^2(\Omega): D^{\alpha}v\in
L^2(\Omega),|\alpha|\le m\}$ where $\alpha$ represents non-negative triple
index. Let $H^1_0(\Omega)$ be the subspace of $H^1(\Omega)$ whose
functions have zero traces on $\Gamma$. We define the spaces of functions having square integrable curl by
\begin{eqnarray}
\Hcurl&:=&\{\Bv\in\Ltwov\,:\;\curl\Bv\in \Ltwov\}, \nonumber\\
\zbHcurl&:=&\{\Bv\in\Hcurl\,:\;\Bn\times\Bv=0\;\;
\hbox{on}\;\Gamma\}, \nonumber
\end{eqnarray}
which are equipped with the following inner product and norm
\begin{equation*}
(\Bv,\Bw)_{\Hcurl}:=(\Bv,\Bw)+(\curl\Bv,\curl\Bw), \;\;
\NHcurl{\Bv}:=\sqrt{(\Bv,\Bv)_{\Hcurl}}\;.
\end{equation*}
Here $\Bn$ denotes the unit outer normal to $\Gamma$. We also use the usual Hilbert space $\BH(\Div,\Omega)$ indicating square integrable divergence.
Denote  by
\[\VD=\BH(\Div,\Omega),~\mathrm{S}=L^2(\Omega),~\VC=\zbHcurl,~\mathrm{R}=H_0^1(\Omega)\]

Next we introduce the continuous mixed variational formulation for the $\mathrm{CPP^3}$-model \eqref{eq:steady-mhd}.
For well posedness we also introduce an extra Lagrange multiplier $r \in H_0^1(\Omega)$ as in\cite{sch04}.
\begin{center}
\fbox{
\parbox{0.90\textwidth}{
Find $\BJ\in \VD$, $\phi \in \mathrm{S}$, $\BA \in \VC$ and
$r \in  \mathrm{R}$ such that the following weak formulation holds
\begin{subequations}\label{weaka}
\begin{align}
\eta(\BJ,\varphibf) - (\phi,\Div\varphibf) - (\Bw\times\curl\BA, \varphibf)   &= (\Bf, \varphibf),\label{weaka:J}\\
-(\Div\BJ, \psi) &= 0, \label{weaka:Phi} \\
-(\BJ, \Ba) + \nu_m(\curl\BA, \curl\Ba) + (\nabla r, \Ba) &= (\Bg, \Ba), \label{weaka:A}\\
(\BA, \nabla s) &= 0. \label{weaka:r}
\end{align}
\end{subequations}
for any $(\varphibf, \psi, \Ba, s) \in \VD \times \mathrm{S} \times \VC \times \mathrm{R}$.
}
}
\end{center}
We refer to the weak formulation \eqref{weaka} by $\mathrm{CPP^3M}$ formulation, where "M" means multiplier.
Based on \eqref{weaka}, we introduce the bilinear forms
\[a_1(\BJ,\varphibf)=\eta\int_{\Omega} \BJ\cdot\varphibf,\qquad a_2(\BA,\Ba) = \nu_m\int_{\Omega}\curl\BA\cdot\curl\Ba\]
\[d_1(\BJ,\phi)=-\int_{\Omega}\phi\Div\BJ,\qquad d_2(\BA,r)=\int_{\Omega} \BA\cdot\nabla r\]
and the trilinear form
\[c(\Bw;\BA,\varphibf)= - \int_{\Omega} \Bw\times\curl\BA\cdot\varphibf\]

Next we state the energy's law for the continuous variational formulation. It is summarized in the following theorem.
\begin{theorem}[Energy's Law]
The solutions of \eqref{weaka} satisfy the energy relation
\begin{subequations}\label{eq:law}
\begin{align}
\eta \|\BJ\|_{\BL^2}^2 & = (\Bf, \BJ) - (\BJ\times\BB,\Bw),\\
\nu_m\|\BB\|_{\BL^2}^2 & = (\Bg, \BA) + (\BJ, \BA).
\end{align}
\end{subequations}
with $\BB = \curl\BA \in \BH_0(\Div, \Omega)$ and $\|\Div\BJ\|_{L^2} = 0$.
\end{theorem}

\begin{proof}
Firstly, let $\Ba = \nabla r$ in \eqref{weaka:A} and note that $r$ has zero boundary conditions, then $r\equiv0$ in $\Omega$.
Let $\varphibf = \BJ, \psi = - \phi$ and $\Ba = \BA$ in \eqref{weaka}, one obtains
\begin{equation}\label{energyJ}
\eta\|\BJ\|_{\BL^2}^2 - (\Bw\times\BB, \BJ) = (\Bf, \BJ)
\end{equation}
and
\begin{equation}\label{energyA}
-(\BJ, \BA) + \nu_m(\curl\BA, \curl\BA) = (\Bg, \BA)
\end{equation}
Denote by
\[\BB = \curl\BA \in \BH_0(\Div,\Omega)\]
Note that $\Div\BJ \in L^2(\Omega)$ and \eqref{weaka:Phi}, then $\|\Div\BJ\|_{L^2} = 0$.
This completes the proof.$\hfill{}$
\end{proof}

\section{Discrete scheme and preconditioning}
Let $\mathcal{T}_h$ be a shape-regular tetrahedral triangulation of $\Omega$, with $h$ the grid size if the partition is quasi-uniform.
We will use finite element spaces which are all conforming, namely
\begin{equation*}
\VD_h \subseteq \VD,~~\mathrm{S}_h \subseteq \mathrm{S},~~\VC_h \subseteq \VC,~~\mathrm{R}_h \subseteq \mathrm{R}
\end{equation*}
For $\VD_h$ we use the $\BH(\Div,\Omega)$-conforming piecewise linear finite element in the second family\cite{xin2013}
\[\VD_h = \{\varphibf_h\in \BH(\Div,\Omega): \varphibf_h|_K\in\BP_1(K), ~ K\in \mathcal{T}_h\}\]
For $\mathrm{S}_h$ we use the piecewise constants finite element
\[\mathrm{S}_h = \{\psi_h\in L^2(\Omega): \psi_h|_K\in P_0(K), ~ K\in \mathcal{T}_h\}\]
The finite element for $\BA$ is the first order N\'{e}d\'{e}lec edge element space in the second family\cite{ne86}
\[\VC_h = \{\Ba_h\in \zbHcurl: \Ba_h|_K\in\BP_1(K),~ K\in \Ct_h\}\]
The finite element space $\mathrm{R}_h$ is defined by
\[\mathrm{R}_h = \{s_h \in H_0^1(\Omega): s_h|_K \in P_2(K),~K \in \Ct_h\}\]
\subsection{A novel mixed finite element scheme}
In this subsection, we will present a mixed finite element scheme to solve the continuous $\mathrm{CPP^3M}$ formulation \eqref{weaka}.
The discrete scheme reads
\begin{center}
\fbox{
\parbox{0.90\textwidth}{
Find $\BJ_h\in \VD_h$, $\phi_h \in \mathrm{S}_h$, $\BA_h \in \VC_h$ and
$r_h \in  \mathrm{R}_h$ such that the following weak formulation holds
\begin{subequations}\label{full-dis}
\begin{align}
a_1(\BJ_h,\varphibf_h)+ d_1(\varphibf_h, \phi_h) + c(\Bw;\BA_h, \varphibf_h) &= (\Bf, \varphibf_h),\label{full-dis:J}\\
d_1(\BJ_h, \psi_h) &= 0,\label{full-dis:Phi}\\
-(\BJ_h, \Ba_h) + a_2(\BA_h, \Ba_h) + d_2(\Ba_h, r_h)                        &= (\Bg, \Ba_h),\label{full-dis:A}\\
d_2(\BA_h, s_h)    &= 0.\label{full-dis:r}
\end{align}
\end{subequations}
for any $(\varphibf_h, \psi_h, \Ba_h, s_h) \in \VD_h \times \mathrm{S}_h \times \VC_h \times \mathrm{R}_h$.
}
}
\end{center}
Because $\Div \VD_h \subseteq \mathrm{S}_h$ and $\BB_h = \curl\BA_h$, we naturally have the following desired precisely divergence-free conditions or
the nice conservative properties
\begin{equation}\label{eq:div-free}
\Div\BJ_h = 0, \quad \Div\BB_h = 0,\quad \mathrm{in}~~\Omega.
\end{equation}
For $\phi_h$ we use piecewise constants approximations, and $\BA_h$ linear approximations, thus we cannot use the transform \eqref{eq:trans} to
recover the original electrical field $\BE_h$. Instead, using Ohm's law is a smart way to do this, namely
\begin{equation}\label{eq:reE}
\BE_h = \eta\BJ_h - \Bw\times\curl\BA_h
\end{equation}
where $\eta$ is the resistivity. In this way, we have obtained all the desired field quantities $\BJ_h, \BB_h$ and $\BE_h$ with the same accuracy.
As a by-product, we also have the electric potential $\phi_h$ with piecewise constants approximation.
We emphasize that the discrete scheme \eqref{full-dis} also has a energy's law which is similar to the continuous case.
See the following theorem.
\begin{theorem}[Discrete Energy's Law]
The solutions of \eqref{full-dis} satisfy the energy relation,
\begin{subequations}\label{eq:lawh}
\begin{align}
\eta \|\BJ_h\|_{L^2}^2 & = (\Bf, \BJ_h) - (\BJ_h\times\BB_h,\Bw),\\
\nu_m\|\BB_h\|_{L^2}^2 & = (\Bg, \BA_h) + (\BJ_h, \BA_h).
\end{align}
\end{subequations}
with $\BB_h = \curl\BA_h \in \BH_0(\Div,\Omega)$ and $\|\Div\BJ_h\|_{L^2} = 0$.
\end{theorem}
\begin{proof}
The proof is similar to the continuous case. We omit it here.
\end{proof}
\vspace{5mm}

\textbf{Remark 2.} The equation \eqref{eq:lawh} can be seen as the discretization of \eqref{eq:law}. And since $\Div\BJ_h$
and $\Div\BB_h$ are precisely zero in $\Omega$, the distinctive feature of our new finite element scheme \eqref{full-dis}
is that it inherits all the conservative properties of the original continuous problem \eqref{eq:mhd}.

\subsection{Block preconditioning method}
After finite element discretization, we will get the linear algebraic system
\begin{equation}\label{Axb}
\mathcal{A}\Vx = \Vb
\end{equation}
where the vector $\Vx$ consists of the degrees of freedom for $(\BJ_h, \phi_h, \BA_h, r_h)$.
The matrix $\mathcal{A}$ could be written in the following block form
\begin{equation*}
\mathcal{A}=\left(
              \begin{array}{cccc}
              \BM   &G^T      &\BK   &0\\
              G     &0        &0     &0\\
              \BX   &0        &\BF   &B^T\\
              0     &0        &B     &0\\
              \end{array}
            \right)
\end{equation*}
where
\begin{align*}
&\BM_{ij}=\eta(\varphibf_j, \varphibf_i), \quad \forall \varphibf_i, \varphibf_j \in \VD_h\\
&G^T_{ij}=-(\psi_j, \Div\varphibf_i), \quad \forall \varphibf_i \in \VD_h, \psi_j \in \mathrm{S}_h \\
&\BK_{ij}=(\curl\Ba_j\times\Bw,\varphibf_i), \quad \forall \varphibf_i \in \VD_h, \forall\Ba_j \in \VC_h\\
&\BX_{ij}=-(\varphibf_j, \Ba_i), \quad \forall \Ba_i \in \VC_h, \forall\varphibf_j \in \VD_h\\
&\BF_{ij}=\nu_m(\curl\Ba_j, \curl\Ba_i), \quad \forall \Ba_i, \Ba_j \in \VC_h\\
&B^T_{ij}=(\nabla s_j, \Ba_i), \quad \forall \Ba_i \in \VC_h, \forall  s_j \in \mathrm{R}_h
\end{align*}
For multi-physics problems, block preconditioning is famous. Now we attempt to give an efficient block preconditioner to devise an iterative solver for the
complicated saddle systems. Before we show our preconditioner, we point out that in\cite{phil2014},
Phillips and Elman constructed an efficient block preconditioner for steady MHD kinematics equations
with $\BB$ and an extra multiplier $r$ as variables. Their basic model is a sub-block of the model in\cite{sch04}, so $\BB_h$ is only weakly divergence free
in their work.

Firstly, we note that the matrix $\BX$ is zero-order term, so maybe we can drop it directly and study preconditioner for
the following modified matrix
\begin{equation}\label{mat:A1}
\mathcal{A}_1=\left(
              \begin{array}{cccc}
              \BM &G^T      &\BK   &0\\
              G   &0        &0     &0\\
              0   &0        &\BF   &B^T\\
              0   &0        &B     &0\\
              \end{array}
            \right)
\end{equation}
For \eqref{mat:A1}, we only need to consider the block preconditioner for the following two saddle systmes
\begin{equation}\label{mat:JA}
\mathcal{A}_J=\left(
  \begin{array}{cc}
  \BM   &G^T \\
  G     &0   \\
  \end{array}
\right),\quad \mathcal{A}_a = \left(
                \begin{array}{cc}
                \BF  &B^T \\
                B    &0   \\
                \end{array}
              \right).
\end{equation}
From the reference\cite{gre07, vassi1996, ma11, ki10}, the following choices should be good candidates
\begin{equation}\label{pre:JA}
\left(
  \begin{array}{cc}
  \widehat{\BM}  &0 \\
  0           &\sigma Q           \\
  \end{array}
\right),\quad \left(
                \begin{array}{cc}
                \widehat{\BF}  &0 \\
                0           &L \\
                \end{array}
              \right).
\end{equation}
where
\[\widehat{\BM}_{ij}:=\sigma^{-1}(\varphibf_j, \varphibf_i) + \sigma^{-1}(\Div\varphibf_j,\Div\varphibf_i),\quad Q_{ij}:=(\psi_j,\psi_i)\]
\[\widehat{\BF}_{ij}:=\textsf{Rm}^{-1}(\curl\Ba_j, \curl\Ba_i) + (\Ba_j, \Ba_i),\quad L_{ij}:=(\nabla s_j, \nabla s_i)\]
However, here we will give some further modifications to improve \eqref{pre:JA}. These improvements are based on augmentation, approximate
block decompositions and the commutativity of the underlying continuous operators.

We start by considering the preconditioner for $(\BJ_h, \phi_h)$ part. Note that ($\eta = \sigma^{-1}$)
\begin{equation}\label{AL:J}
\left(
    \begin{array}{cc}
      I & \eta G^T Q^{-1} \\
      0 & I \\
    \end{array}
  \right)\left(
           \begin{array}{cc}
           \BM & G^T \\
            G & 0 \\
           \end{array}
         \right) = \left(
                     \begin{array}{cc}
                       \BM + \eta G^T Q^{-1} G  & G^T \\
                       G & 0 \\
                     \end{array}
                   \right)
\end{equation}
and
\[\left(
  \begin{array}{cc}
  \BM + \eta G^T Q^{-1} G  & G^T \\
  G & 0 \\
  \end{array}
\right)=\left(
          \begin{array}{cc}
            I & 0 \\
            G\widetilde{\BM}^{-1} & I \\
          \end{array}
        \right)\left(
                 \begin{array}{cc}
                   \widetilde{\BM} & G^T \\
                   0               & - G\widetilde{\BM}^{-1} G^T \\
                 \end{array}
               \right) \triangleq \mathcal{L}_J\mathcal{U}_J
\]
with $\widetilde{\BM} = \BM + \eta G^T Q^{-1} G$.
The underlying continuous operator of $\widetilde{\BM} $ is
\[\eta \mathrm{Id} - \eta \nabla\Div\]
so $\widehat{\BM}$ should be an ideal approximation for $\widetilde{\BM}$.
The operator of $G\widetilde{\BM}^{-1} G^T$ is as follows
\[-\Div(\eta \mathrm{Id}- \eta\nabla\Div)^{-1}\nabla\]
Because Laplace operator can commutate with the gradient opterator
\[(\eta \mathrm{Id} - \eta\nabla\Div)\nabla = \eta\nabla(\mathrm{Id} - \Delta)_\phi
\Rightarrow \sigma \nabla(\mathrm{Id} - \Delta)_\phi^{-1} = (\eta \mathrm{Id} - \eta\nabla\Div)^{-1}\nabla\]
we obtain
\[-\Div(\eta \mathrm{Id} - \eta\nabla\Div)^{-1}\nabla = -\sigma\Div\nabla (\mathrm{Id} - \Delta)_\phi^{-1} \approx  \sigma \mathrm{Id}\]
From above we can approximate $\mathcal{U}_J$ be
\begin{equation}\label{preJ1}
\left(
                 \begin{array}{cc}
                   \widehat{\BM} & G^T \\
                   0 & -\sigma Q \\
                 \end{array}
               \right)
\end{equation}
Due to \eqref{AL:J} and \eqref{preJ1},
\begin{equation}\label{preJ2}
\mathcal{P}_J = \left(
    \begin{array}{cc}
    I & -\eta G^T Q^{-1} \\
    0 & I \\
    \end{array}
  \right)\left(
  \begin{array}{cc}
  \widehat{\BM} & G^T \\
  0 & -\sigma Q \\
  \end{array}
  \right) = \left(
            \begin{array}{cc}
            \widehat{\BM} & 2G^T \\
            0 & -\widehat{Q} \\
            \end{array}
            \right)
\end{equation}
should be a good preconditioner for $\mathcal{A}_J$ with $\widehat{Q} = \sigma Q$.

Next we consider the block preconditioning for $(\BA_h, r_h)$ part. Note that
\begin{equation}\label{AL:A}
\left(
    \begin{array}{cc}
    I & B^T L^{-1} \\
    0 & I \\
    \end{array}
  \right)\left(
           \begin{array}{cc}
           \BF & B^T \\
           B & 0 \\
           \end{array}
         \right) = \left(
                     \begin{array}{cc}
                     \BF + B^T L^{-1} B  & B^T \\
                     B & 0 \\
                     \end{array}
                   \right)
\end{equation}
and
\[\left(
  \begin{array}{cc}
  \BF + B^T L^{-1} B  & B^T \\
  B & 0 \\
  \end{array}
\right)=\left(
          \begin{array}{cc}
          I & 0 \\
          B\widetilde{\BF}^{-1} & I \\
          \end{array}
        \right)\left(
                 \begin{array}{cc}
                 \widetilde{\BF} & B^T \\
                 0 & - B\widetilde{\BF}^{-1} B^T \\
                 \end{array}
               \right) \triangleq \mathcal{L}_a\mathcal{U}_a
\]
with $\widetilde{\BF} = \BF + B^T L^{-1} B$. As above we want to give a good matrix approximation for $\mathcal{U}_a$.
From the reference\cite{gre07}, $\widehat{\BF}$ is good approximation of $\widetilde{\BF}$. Consider the continuous operator
of $B\widetilde{\BF}^{-1} B^T$
\[-\Div[\nu_m\curl\curl + \nabla(-\Delta)^{-1}(-\Div)]^{-1}\nabla\]
Note that
\[[\nu_m\curl\curl + \nabla(-\Delta)^{-1}(-\Div)]\nabla = \nabla\]
Thus one have
\[[\nu_m\curl\curl + \nabla(-\Delta)^{-1}(-\Div)]^{-1}\nabla = \nabla\]
and further
\[-\Div[\nu_m\curl\curl + \nabla(-\Delta)^{-1}(-\Div)]^{-1}\nabla = -\Div\nabla = -\Delta_r\]
From the above equation, the matrix $L$ should be a good approximation for $B\widetilde{\BF}^{-1} B^T$.
And a reasonable approximation of $\mathcal{U}_a$ is as follows
\begin{equation}\label{preA1}
\left(
\begin{array}{cc}
\widehat{\BF} &B^T \\
0             &-L  \\
\end{array}
\right)
\end{equation}
And due to \eqref{AL:A}
\begin{equation}\label{preA2}
\mathcal{P}_a = \left(
    \begin{array}{cc}
      I & -B^T L^{-1} \\
      0 & I           \\
    \end{array}
  \right)\left(
  \begin{array}{cc}
  \widehat{\BF} & B^T \\
  0             & -L \\
  \end{array}
  \right) = \left(
  \begin{array}{cc}
  \widehat{\BF} & 2B^T \\
  0             & -L   \\
  \end{array}
  \right)
\end{equation}
should be a good preconditioner for $\mathcal{A}_a$.

Then we get our preconditioner (right-preconditioning) for $\mathcal{A}$
provided that $\mathcal{A}_1$ in \eqref{mat:A1} is a good preconditioner for $\mathcal{A}$
\begin{equation}\label{pre1:A}
\mathcal{P} =\left(
              \begin{array}{cccc}
              \widehat{\BM}   &2G^T         &\BK           &0    \\
              0               &-\widehat{Q} &0             &0    \\
              0               &0            &\widehat{\BF} &2B^T \\
              0               &0            &0             &-L   \\
              \end{array}
            \right)
\end{equation}
Next we give the preconditioned FGMRES iterative method (Algorithm \ref{alg:gmres}) for practical implementation.
\emph{}For convenience in notation, given a general vector $\Vx$ which has the same size as one column vector of $\mathcal{A}$,
we let $(\Vx_J, \Vx_\phi, \Vx_A, \Vx_r)$ be the vectors which consist of entries of $\Vx$ and correspond to
$(\BJ_h, \phi_h, \BA_h, r_h)$ respectively.

\begin{algorithm}[Preconditioned FGMRES Algorithm]\label{alg:gmres}
{\sf
Given the tolerances $\varepsilon\in (0,1)$ and $\varepsilon_0\in (0,1)$, the maximal number of FGMRES iterations $N>0$, and the
initial guess $\mathbf{x}^{(0)}$ for the solution of \eqref{Axb}. Set $k=0$ and compute the residual vector
\begin{align*}
\mathbf{r}^{(k)} = \mathbf{b}-\mathcal{A}\mathbf{x}^{(k)}.
\end{align*}

\noindent
While $\left(k<N\;\; \&\;\; \N{\mathbf{r}^{(k)}}_2 >\varepsilon\N{\mathbf{r}^{(0)}}_2\right)$ do
\begin{enumerate}
  \item Solve $L \Ve_r = -\Vr^{(k)}_r$ by preconditioned CG method with tolerance $\varepsilon_0$.
  The preconditioner is the algebraic multigrid method (AMG)\cite{hen02}.
  \item Solve $\widehat{\BF}\Ve_A = \Vr^{(k)}_A - 2B^T\Ve_r$ with tolerance $\varepsilon_0$.
  We use preconditioned CG method with the Hiptmair-Xu preconditioner\cite{hip07}.
  \item Solve $\widehat{Q} \Ve_ \phi = -\Vr_{\phi}^{(k)}$ by $2$ CG iterations with the diagonal preconditioning.
  \item Solve $\widehat{\BM} \Ve_J = \Vr^{(k)}_J - 2G^T \Ve_{\phi} - \BK\Ve_A$ using  MUMPS direct solver\cite{mumps}.
  \item Update the solutions in FGMRES iteration to obtain $\mathbf{x}^{(k+1)}$.
  \item Set $k:=k+1$ and compute the residual vector $\mathbf{r}^{(k)} = \mathbf{b}-\mathcal{A}\mathbf{x}^{(k)}$.
\end{enumerate}
End while.}
\end{algorithm}
\vspace{5mm}

\textbf{Remark 3.} In Algorithm \ref{alg:gmres} we mainly concern the outer iterations, so we fix the relative tolerance $\varepsilon_0$ for the inner solvers.
Because iterative methods are used as a preconditioner, we we use FGMRES\cite{saad1993} solver rather than GMRES solver.
For the sub-problem associated with $\widehat{\BM}$, there has already existed efficient preconditioners instead of expensive direct solver here.
We refer to\cite{arnold97} for simple multigrid and domain decomposition preconditioners and \cite{hip07} for auxiliary space preconditioners based
on AMG algorithm for negative Laplace problem. In future we will explore using more efficient inner solvers.

\textbf{Remark 4.} In the continuous case, we have the following equation
\begin{equation}\label{eq:phi1}
\eta\BJ = -\partial_t\BA - \nabla\phi + \Bw \times \curl \BA + \Bf
\end{equation}
Applying the divergence operator to the \eqref{eq:phi1} with $\Div\BA = 0$, we obtain
\[-\Delta\phi + \Div(\Bw\times\curl\BA) + \Div\Bf = 0, \quad\phi = 0~~\hbox{on}\;\;\Gamma\]
Thus given $\Bw$ and $\curl\BA_h$ one can solve the above homogenous elliptic equations to improve the precision of the scalar potential $\phi_h$.
For example, solve the following variational problem using FEM with node element
\begin{equation}
\hbox{Find}~~\phi \in H_0^1(\Omega),\quad \mathrm{s.t.}, \quad(\nabla\phi, \nabla\psi)
= (\Bw\times\curl\BA_h, \nabla\psi) + (\Bf, \nabla\psi), \quad\forall \psi \in H_0^1(\Omega)
\end{equation}
\section{Numerical experiments}
In this section, we present three numerical examples to verify the divergence-free feature of the discrete solutions,
the convergence rate of finite element solutions and the performance of the proposed block preconditioner.
Due to $\BB_h = \curl\BA_h \in \BH(\Div,\Omega)$ we naturally have $\|\Div\BB_h\|_{L^2} = 0$.
So we will only report the divergence of the discrete current density $\BJ_h$.

The computational domain $\Omega$ is unit cube $(0,1)^3$.
The code is developed based on the finite element package: \textbf{Parallel Hierarchical Grid (PHG)}\cite{zh09-1,zh09-2}.
We set the tolerances by $\varepsilon = 10^{-10}$ and $\varepsilon_0 = 10^{-3}$ in Algorithm \ref{alg:gmres}.
We use PETSc's FGMRES solver and set
the maximal iteration number by $N=500$\cite{petsc}.

%%%%%%%%%%%%%%%%%%%%%%%%%%%%%%%%%%%%%%%%%
%% space precision %%%%%%
%%%%%%%%%%%%%%%%%%%%%%%%%%%%%%%%%%%%%%%%%
\begin{example}[Precision test I: $\Bw\equiv \mathbf{0}$]\label{ex1}
Set $\sigma = {\sf Rm} = 1$, and use the following analytic solutions
\[\BJ = (\sin y, 0, x^2),~\phi = z,~\BA = (0, \cos x , 0),~r=0\]
\end{example}
\begin{table}[!htb]
  \centering
  \caption{Convergence rate of the finite element solutions (Example \ref{ex1}).}
  \label{eddyEnergyError}
  \begin{tabular*}{14cm}{@{\extracolsep{\fill}}|c|c|c|c|c|c|c|}
  \hline
  $h$    &$\|\BJ-\BJ_h\|_{\BL^2}$ &order  &$\|\phi - \phi_h\|_{L^2}$  &order &$\|\BA-\BA_h\|_{\BH(\curl)}$ &order\\ \hline
  0.8660 &1.5041e-02              &---    &1.0206e-01   &---    &9.8055e-02 &---    \\    \hline
  0.4330 &3.7629e-03              &1.9990 &5.1031e-02   &1.000  &4.8103e-02 &1.0275 \\    \hline
  0.2165 &9.4089e-04              &1.9997 &2.5516e-02   &1.000  &2.3779e-02 &1.0164 \\    \hline
  0.1083 &2.3523e-04              &2.0000 &1.2758e-02   &1.000  &1.1821e-02 &1.0083 \\    \hline
  \end{tabular*}
\end{table}

\begin{table}[!htb]
  \centering
  \caption{Divergence of $\BJ_h$ and $\BL^2$ error of $\BA_h$ (Example \ref{ex1}).}
  \label{eddyL2Error}
  \begin{tabular*}{10cm}{@{\extracolsep{\fill}}|c|c|c|c|}
  \hline
  $h$     &$\|\Div\BJ_h\|_{L^2}$   &$\|\BA-\BA_h\|_{\BL^2}$ &order   \\\hline
  0.8660  &\textbf{3.2878e-12}              &1.2353e-02              &---     \\\hline
  0.4330  &\textbf{1.4178e-11}              &3.2468e-03              &1.9278  \\\hline
  0.2165  &\textbf{4.5646e-11}              &8.1712e-04              &1.9904  \\\hline
  0.1083  &\textbf{5.8225e-11}              &2.0336e-04              &2.0065  \\\hline
  \end{tabular*}
\end{table}

From Table \ref{eddyEnergyError} and Table \ref{eddyL2Error}, we know that the following optimal convergence rates are obtained
\begin{equation}
\|\BJ-\BJ_h\|_{\BL^2} \sim O(h^2), ~~ \|\phi - \phi_h\|_{L^2} \sim O(h).
\end{equation}
\begin{equation}
\|\BA-\BA_h\|_{\BH(\curl)} \sim O(h),~~ \|\BA-\BA_h\|_{\BL^2} \sim O(h^2).
\end{equation}
if the the prescribed velocity field $\Bw$ is identically zero in $\Omega$. And the divergence of $\BJ_h$ is almost zero.

%%%%%%%%%%%%%%%%%%%%%%%%%%%%%%%%%%%%%%%%
%% space precision
%%%%%%%%%%%%%%%%%%%%%%%%%%%%%%%%%%%%%%%%
\begin{example}[Precision test II: $\Bw = (x,y,z)$]\label{ex2}
Set $\sigma = {\sf Rm} = 1$, and use the same analytic solutions as Example \ref{ex1}
\[\BJ = (\sin y, 0, x^2),~\phi = z,~\BA = (0, \cos x , 0),~r=0\]
\end{example}
\begin{table}[!htb]
  \centering
  \caption{Convergence rate of the finite element solutions (Example \ref{ex2}).}
  \label{EnergyError}
  \begin{tabular*}{14cm}{@{\extracolsep{\fill}}|c|c|c|c|c|c|c|}
  \hline
  $h$    &$\|\BJ-\BJ_h\|_{\BL^2}$   &order &$\|\phi - \phi_h\|_{L^2}$  &order &$\|\BA-\BA_h\|_{\BH(\curl)}$ &order\\ \hline
  0.8660 &5.9797e-02   &---     &1.0208e-01   &---     &9.8057e-02  &---     \\\hline
  0.4330 &2.6435e-02   &1.1776  &5.1034e-02   &1.0002  &4.8104e-02  &1.0275  \\\hline
  0.2165 &1.2526e-02   &1.0775  &2.5516e-02   &1.0001  &2.3780e-02  &1.0164  \\\hline
  0.1083 &6.1235e-03   &1.0325  &1.2758e-02   &1.0000  &1.1821e-02  &1.0084  \\\hline
  \end{tabular*}
\end{table}

\begin{table}[!htb]
  \centering
  \caption{Divergence of $\BJ_h$ and $\BL^2$ error of $\BA_h$ (Example \ref{ex2}).}
  \label{L2Error}
  \begin{tabular*}{8cm}{@{\extracolsep{\fill}}|c|c|c|c|}
  \hline
  $h$       &$\|\Div\BJ_h\|_{L^2}$   &$\|\BA-\BA_h\|_{\BL^2}$ &order    \\\hline
  0.8660    &\textbf{6.3558e-11}              &1.2198e-02              &---      \\\hline
  0.4330    &\textbf{1.5463e-11}              &3.2073e-03              &1.9272   \\\hline
  0.2165    &\textbf{5.5790e-11}              &8.0656e-04              &1.9915   \\\hline
  0.1083    &\textbf{6.7967e-11}              &2.0070e-04              &2.0067   \\\hline
  \end{tabular*}
\end{table}

From Table \ref{EnergyError} and Table \ref{L2Error}, we know that optimal convergence rates for $\phi, \BA$
are obtained. But the convergence rate of $\BJ_h$ is only one order, which is not optimal. The reason is
that $\curl\BA_h$ only has 1 order precision, if the induced electrical field $\Bw\times\curl\BA_h$ is not zero.
To get the optimal convergence rate for the current density, we can use the second order edge element to
discretize $\BA$. Moreover, the divergence of $\BJ_h$ is also almost zero, which indicates the divergence-free
feature of our FE scheme \eqref{full-dis}.

Because in the present paper, we mainly focus on the conservative properties of the discrete current density $\BJ_h$
and magnetic induction $\BB_h$
\[\Div\BJ_h = 0, \quad \Div\BB_h = 0.\]
and the convergence of the solutions. The higher order element discretization
for vector potential $\BA$ will be the future investigation.

%%%%%%%%%%%%%%%%%%%%%%%%%%%%%%%%%%%%
%%%%%%%%%%%%%%%%%%%%%%%%%%%%%%%%%%%%
\begin{example}[Performance of the block preconditioners]\label{ex3}
Set $\sigma = 1$, and prescribe the velocity field by
\begin{equation}
\Bw = \left(
        \begin{array}{c}
          -16x(1-x)y(1-y)\sin \theta \\
          16x(1-x)y(1-y)\cos \theta \\
          0 \\
        \end{array}
      \right)
\end{equation}
where $\theta$ is the angle between vector $(x,y,0)$ and the positive direction of $x$-axis.
And we know that $\Bw\cdot\Bn \equiv 0$ on the boundary. The applied magnetic field $\BB_s = (1,0,0)$ with $\BA_s = (0, 0, y)$. The source terms
\[\Bf = \Bw\times\BB_s, \quad \Bg = - \nu_m\curl\BB_s  = \mathbf{0}.\]
and the boundary conditions
\[\phi = 0,\quad \BA\times\Bn = \mathbf{0},\quad \mathrm{on}~~\Gamma.\]
\end{example}
We know that the setting actually corresponds to PEC boundary conditions
for the original systems \eqref{eq:mhd}. The distribution of $\Bw$ on cross section $z=0.5$ is shown in Fig.\ref{fig:wh}.
\begin{figure}[!htb]
  \centering
  \includegraphics[width=0.45\textwidth]{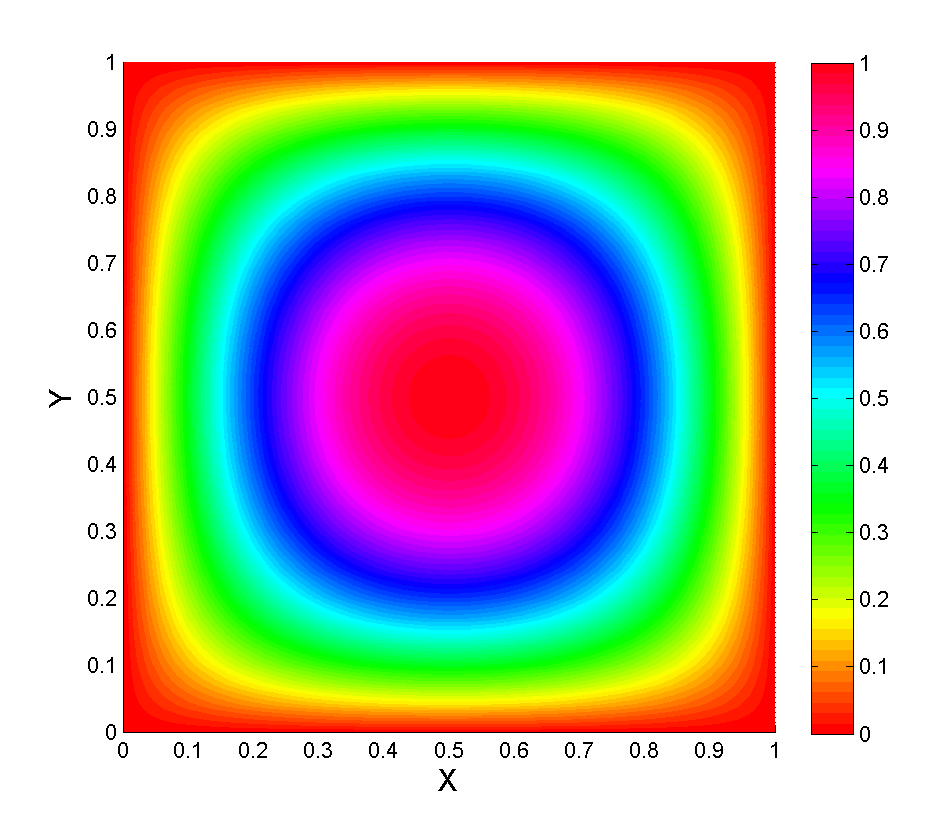}
  \includegraphics[width=0.45\textwidth]{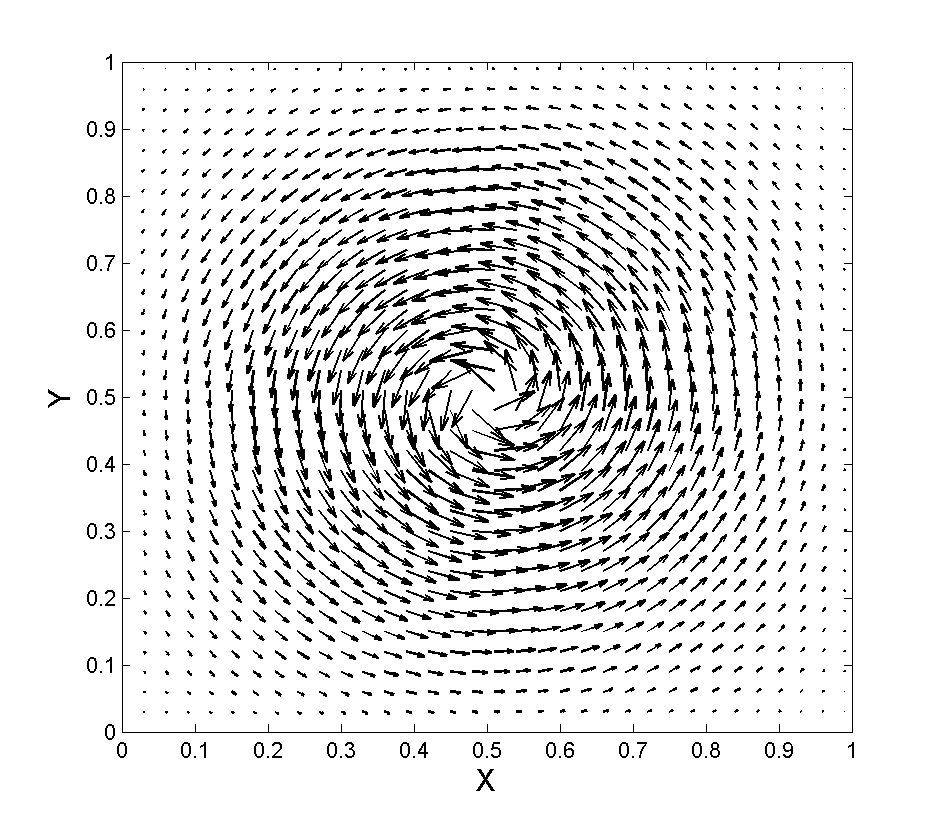}
  \caption{Distribution of $\Bw$ on $z=0.5$.
  Left: magnitude. Right: projected vector.}
  \label{fig:wh}
\end{figure}

In this third example, we want to test the performance of the proposed block preconditioner \eqref{pre1:A}.
The information of grid and degree of freedom is listed in Table \ref{pre:grid}.
We use three different magnetic Reynolds number $\textsf{Rm} = 1, 20, 50$ to show the performance of the preconditioenr
$\mathcal{P}$.

From Table \ref{pre1:tau1}, we observe that the quasi-optimality of $\mathcal{P}$  with respect to the grid refinement.
And the preconditioned FGMRES solver is still robust for relatively high physical parameter $\textsf{Rm} = 50$.
Note that the relative error tolerance for FGMRES is $10^{-10}$.

\begin{table}[!htb]
  \centering
  \caption{The mesh sizes and the numbers of DOFs (Example \ref{ex3}).}
  \label{pre:grid}
  \begin{tabular*}{10cm}{@{\extracolsep{\fill}}|c|c|c|c|}
  \hline
  Mesh        &$h$       &DOFs for $(\BJ,\phi)$  &DOFs for $(\BA,r)$  \\\hline
  $\Ct_1$     &0.8660    &408                    &321                 \\\hline
  $\Ct_2$     &0.4330    &2976                   &1937                \\\hline
  $\Ct_3$     &0.2165    &23286                  &13281               \\\hline
  $\Ct_4$     &0.1083    &176640                 &97985               \\\hline
  \end{tabular*}
\end{table}

\begin{table}[!htb]
  \centering
  \caption{FGMRES iteration number with preconditioner $\mathcal{P}$ (Example \ref{ex3}).}
  \label{pre1:tau1}
  \begin{tabular*}{8cm}{@{\extracolsep{\fill}}|c|c|c|c|}
  \hline
  Mesh        &\textsf{Rm} = 1     &\textsf{Rm} = 20    &\textsf{Rm} = 50\\\hline
  $\Ct_1$     &7      &16    &23   \\\hline
  $\Ct_2$     &8      &21    &39   \\\hline
  $\Ct_3$     &8      &22    &44   \\\hline
  $\Ct_4$     &7      &23    &44   \\\hline
  \end{tabular*}
\end{table}

Using the grid $\Ct_4$ we plot the convergence history of FGMRES method in Fig. \ref{fig:FGMRES}
with different $\textsf{Rm}$. Finally, setting $\textsf{Rm} = 50$, we show the simulation of the steady MHD kinematics
\eqref{eq:steady-mhd} with the grid  $\Ct_4$. Fig. \ref{fig:Bh} shows the distribution of $\BB_h$ on cross section $z=0.5$, from which
we clearly see the perturbation by the flow profile $\Bw$. Fig. \ref{fig:Jhy} shows the distribution of $\BJ_h$ on cross section $y=0.5$,
from which we observed the layered structure. At last Fig. \ref{fig:Jhz} shows the magnitude of $|\BJ_h|$ on cross section $z = 0.5$.
Furthermore
\[\|\Div\BJ_h\|_{L^2(\Omega)} \approx 9.0234 \times 10^{-13},\quad \|\Div\BB_h\|_{L^2(\Omega)} = 0.\]

\begin{figure}[!htb]
  \centering
  \includegraphics[width=0.60\textwidth]{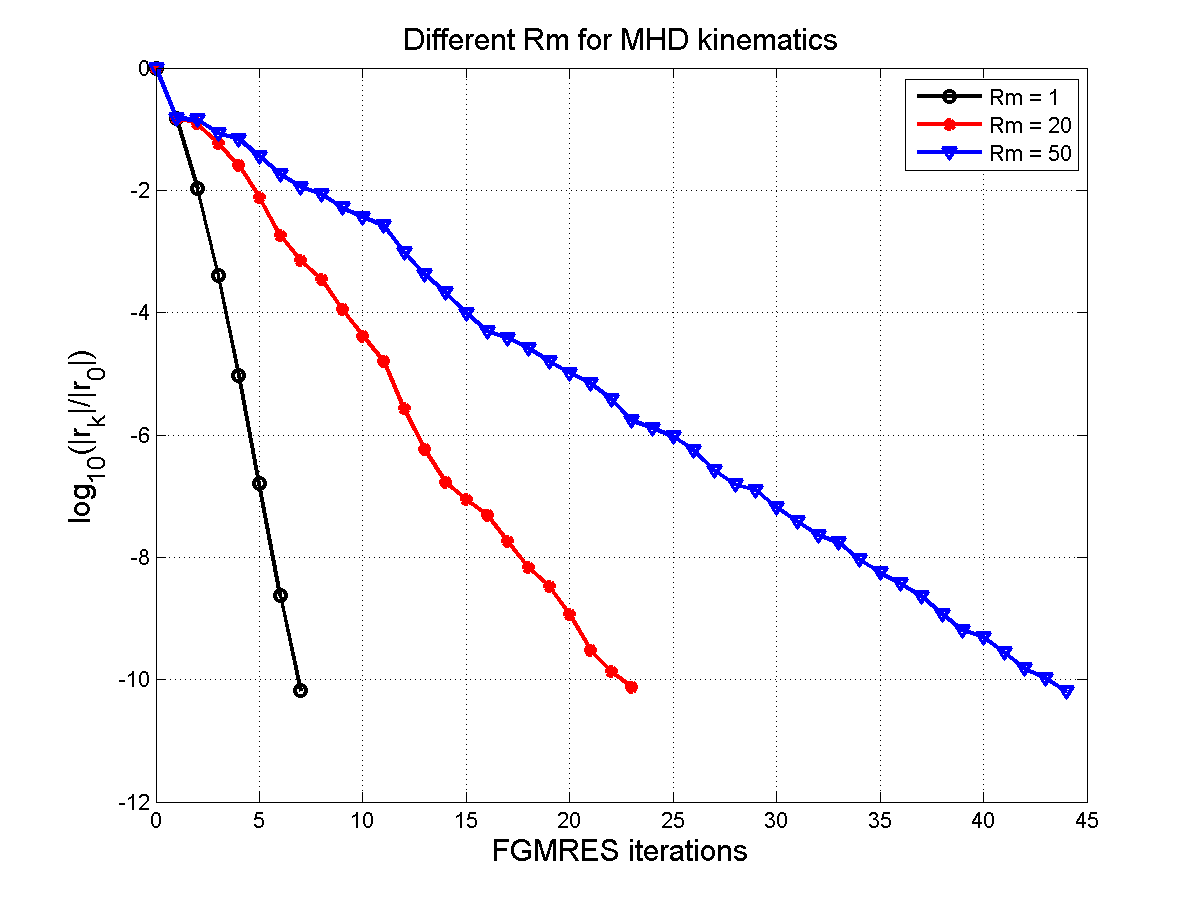}
  \caption{Convergence history of the preconditioned FGMRES solver with different \textsf{Rm} and $\Ct_4$.}
  \label{fig:FGMRES}
\end{figure}

\begin{figure}[!htb]
  \centering
  \includegraphics[width=0.60\textwidth]{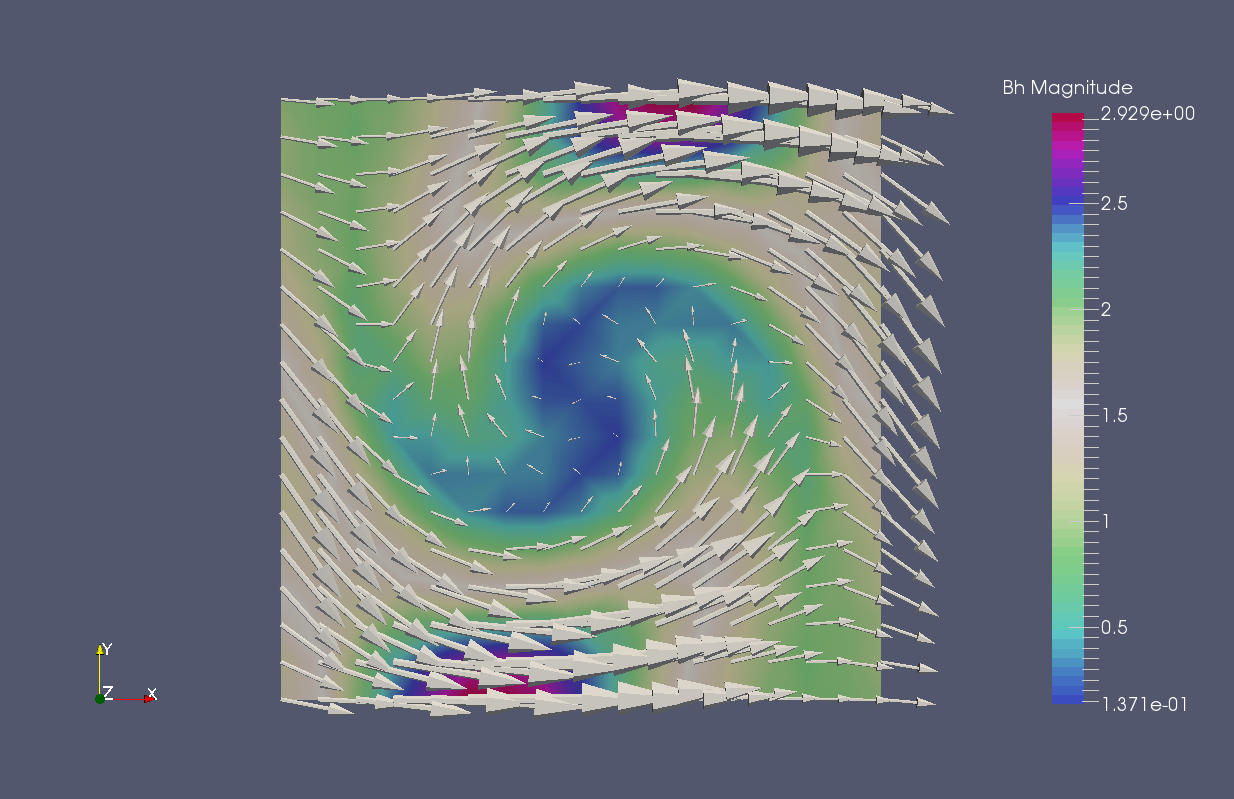}
  \caption{Distribution of $\BB_h$ on cross section $z = 0.5$ with \textsf{Rm} = 50 and $\Ct_4$.}
  \label{fig:Bh}
\end{figure}

\begin{figure}[!htb]
  \centering
  \includegraphics[width=0.60\textwidth]{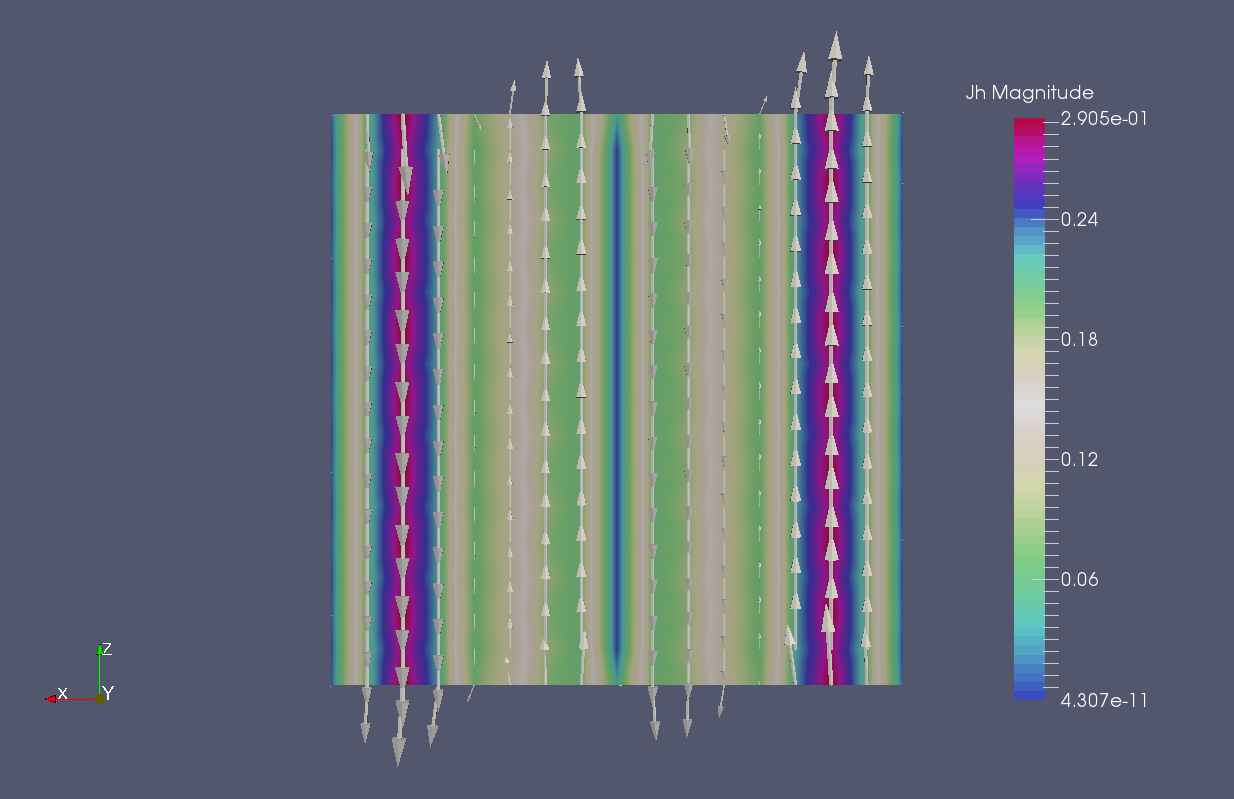}
  \caption{Distribution of $\BJ_h$ on cross section $y = 0.5$ with \textsf{Rm} = 50 and $\Ct_4$.}
  \label{fig:Jhy}
\end{figure}

\begin{figure}[!htb]
  \centering
  \includegraphics[width=0.60\textwidth]{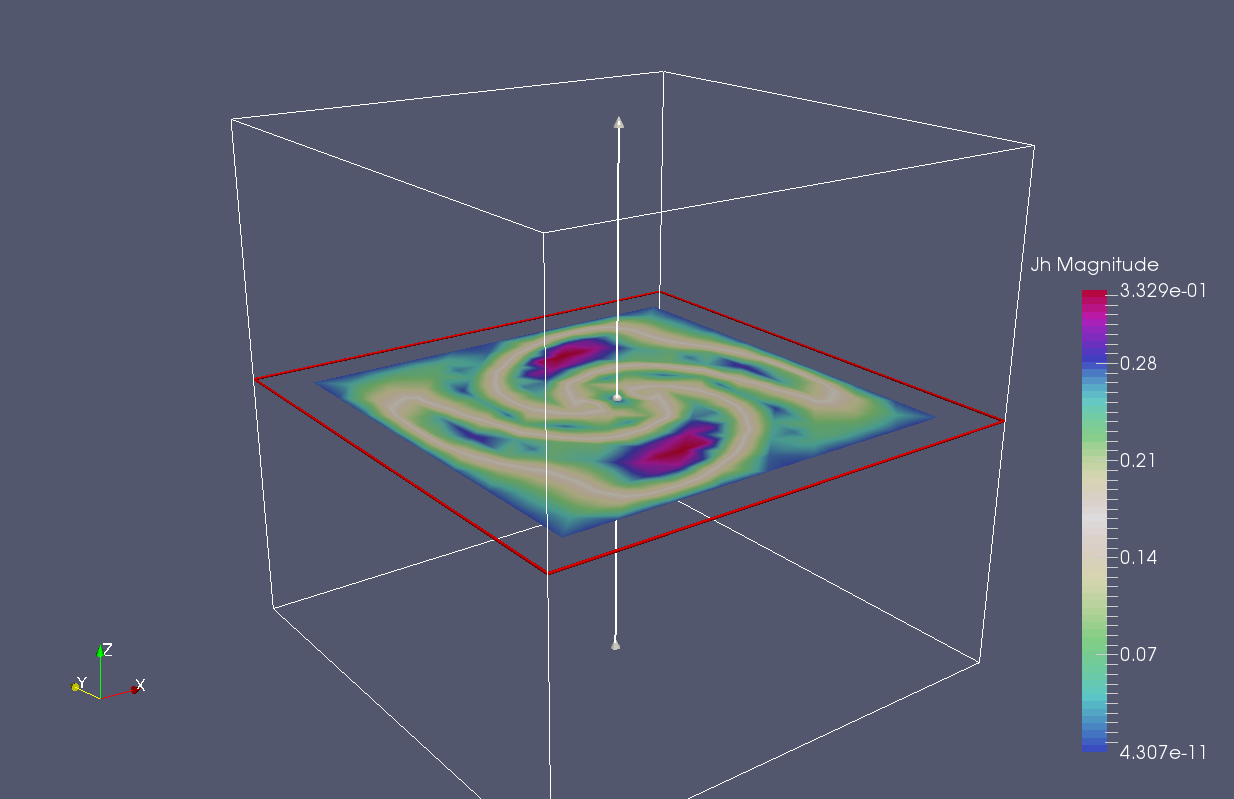}
  \caption{Magnitude of $|\BJ_h|$ on cross section $z = 0.5$ with \textsf{Rm} = 50 and $\Ct_4$.}
  \label{fig:Jhz}
\end{figure}

\section{Conclusions}
In this paper, we develop a new divergence-free finite element method for the MHD kinematics equations which can ensure
exactly divergence-free approximations of the current density and magnetic induction.
Moreover, we devise an efficient preconditioned iterative solver using a block preconditioner. For future research, one could consider the transient case \eqref{eq:new-mhd-dimen},
more robust preconditioner and the coupling with the Navier-Stokes equations.
Considering the uncertainty in the equations as\cite{phil2014} using the FE scheme proposed in this paper
is also interesting.

\newpage
\section*{Acknowledgments}
The author would like to thank his advisor Prof. Wei Ying Zheng for leading him to study the computational methods for MHD
and his guidance throughout the graduate career. The author is indebted to Prof. Dr. Ralf Hiptmair for the fruitful
discussion about the usage of vector potential in MHD model. The author would also like to acknowledge the support of the
the Institute of Computational Mathematics and Scientific/Engineering Computing of Chinese Academy of Sciences, which provided the
computer facilities and other resources for this research,
especially the PHG platform (\textcolor[rgb]{0.00,0.07,1.00}{http://lsec.cc.ac.cn/phg/}) developed by Prof. Lin Bo Zhang.

\bibliographystyle{amsplain}

\end{document}